\documentclass[a4paper]{amsart}
\usepackage{graphicx}
\usepackage{amssymb}
\usepackage{amsmath}
\usepackage{amsthm}
\usepackage{amscd}
\usepackage{color}
\usepackage{colordvi}
\usepackage[all,2cell]{xy}

\UseAllTwocells \SilentMatrices
\newtheorem{thm}{Theorem}[section]

\newtheorem{cor}[thm]{Corollary}
\newtheorem{lem}[thm]{Lemma}
\newtheorem{exm}[thm]{Example}

\newtheorem{prop}[thm]{Proposition}
\theoremstyle{definition}

\theoremstyle{remark}
\newtheorem{rem}[thm]{\bf Remark}
\numberwithin{equation}{section}

\begin{document}
\title[Frobenius functors and Gorenstein homological properties]{Frobenius functors and Gorenstein homological properties}
\author[Xiao-Wu Chen, Wei Ren] {Xiao-Wu Chen, Wei Ren}

\thanks{}
\thanks{}
\subjclass[2010]{18G25, 18G20, 18G80, 16E65}
\date{\today}

\thanks{E-mail: xwchen$\symbol{64}$mail.ustc.edu.cn; wren$\symbol{64}$cqnu.edu.cn}
\keywords{Frobenius functor,  Gorenstein projective object, Gorenstein global dimension}%

\maketitle

\dedicatory{}%
\commby{}%

\begin{abstract}
We prove that  any faithful Frobenius functor between abelian categories preserves the Gorenstein projective dimension of objects. Consequently, it preserves and reflects Gorenstein projective objects. We give conditions on when a Frobenius functor preserves the stable categories of Gorenstein projective objects, the singularity categories and the Gorenstein defect categories, respectively. In the appendix, we give a direct proof of the following known result: for an abelian category with enough projectives and injectives, its global  Gorenstein projective dimension coincides with its global  Gorenstein injective dimension.
\end{abstract}

\section{Introduction}

Frobenius extensions between rings are classical in algebra, which extend the notion of a Frobenius algebra over a field. For example, the group algebra $kG$ of a finite group $G$ over a field $k$ is a Frobenius algebra, and the natural embedding  of an arbitrary ring $R$ into the group ring $RG$ is  a Frobenius extension.

Frobenius bimodules are coordinate-free generalizations of Frobenius extensions   \cite{Kas, NT60}. We mention that Frobenius bimodules appear naturally in stable equivalences of Morita type \cite{DM07, Xi08}.

Frobenius bimodules correspond bijectively to Frobenius functors between module categories \cite{Mor65, CGN99,CDM02}. Therefore, Frobenius functors are viewed as a categorical analogue to Frobenius bimodules. The following example, although not so well known, seems to be fundamental in the homological study of  complexes: the forgetful functor from the category of cochain complexes of modules to the category of graded modules is a Frobenius functor; consult \cite[Subsection~2.2]{Kel} and Example~\ref{exm:dg}.

Gorenstein projective modules are central in relative homological algebra \cite{EJ00, Hol04}. The following fact seems to be well known: for a finite group $G$, a module over $RG$ is Gorenstein projective if and only if its underlying $R$-module is Gorenstein projective; compare \cite[Subsection 8.2]{Buc} and \cite{Chen11}.  This motivates  the work \cite{Ren1, Ren2, Zhao}, where it is proved that Frobenius extensions between rings preserve and reflect Gorenstein projective modules; compare \cite{HLGZ}. Similarly, the forgetful functor from the category of cochain complexes to the category of graded modules preserves and reflects Gorenstein projective objects \cite{Yang11, YL11, Ren1}. This similarity motivates the following natural question: does  any Frobenius functor between two abelian categories preserve and reflect Gorenstein projective objects?

The first result  gives an affirmative answer to this question, and thus unifies the above mentioned results, provided that  the Frobenius functor is faithful; see Theorem~\ref{thm:GP}. As an advantage of our categorical consideration, a similar result for Gorenstein injective objects follows by categorical duality. Indeed, our result is slight stronger, as it is shown that a faithful Frobenius functor preserves the Gorenstein projective dimension of objects.

For any ring $R$, there are three triangulated categories measuring its homological singularity, namely, the stable category of Gorenstein projective $R$-modules, the singularity category \cite{Buc, Or04} and the Gorenstein defect category \cite{BJO}. In a similar manner, these triangulated categories are defined for an abelian category with enough projective objects. It is known that under mild conditions, a stable equivalence of Morita type preserves these triangulated categories \cite{LX, ZZ}. As mentioned above, a stable equivalence of Morita type gives rise to a Frobenius functor between module categories. Therefore, it is very natural to expect that a certain Frobenius functor preserves these three triangulated categories. We confirm this expectation in the second result; see Proposition~\ref{prop:tri-equ}.

The structure of this paper is straightforward. In the appendix, we give a direct proof of the following known result \cite{BR}: for an abelian category with enough projective objects and enough injective objects,  the supermum of the Gorenstein projective dimension of all objects coincides with the supermum of the Gorenstein injective dimension of all objects. We mention that the original proof in \cite[Chapter~VII]{BR} seems to be very indirect and  nontrivial to follow.

\section{Adjoint pairs and Frobenius pairs}

In this section, we recall standard facts and examples on Frobenius functors.

Throughout, we assume that both $\mathcal{A}$ and $\mathcal{B}$ are abelian categories with enough projective objects. Denote by $\mathcal{P}(\mathcal{A})$ and $\mathcal{P}(\mathcal{B})$ the full subcategories of projective objects in $\mathcal{A}$ and $\mathcal{B}$, respectively.

Let $F\colon \mathcal{A}\rightarrow \mathcal{B}$ be an additive functor with a right adjoint $G\colon \mathcal{B}\rightarrow \mathcal{A}$. We will denote  the adjoint pair $(F, G)$ by $F\colon \mathcal{A}\rightleftarrows \mathcal{B}\colon G$. We denote the unit by $\eta\colon {\rm Id}_\mathcal{A}\rightarrow GF$, and the counit by $\varepsilon\colon FG\rightarrow {\rm Id}_\mathcal{B}$.

The following results are standard. We include a complete proof for the convenience of the reader. For a class $\mathcal{S}$  of objects in $\mathcal{A}$, we denote by ${\rm add}\; \mathcal{S}$ the full subcategory consisting of direct summands of finite direct sums of objects in $\mathcal{S}$.

\begin{lem}\label{lem:adj}
Let $F\colon \mathcal{A}\rightleftarrows \mathcal{B}\colon G$ be an adjoint pair. Then the following statements hold.
\begin{enumerate}
\item $F(\mathcal{P}(\mathcal{A}))\subseteq \mathcal{P}(\mathcal{B})$ if and only if the functor $G$ is exact;
\item ${\rm add}\; F(\mathcal{P}(\mathcal{A}))=\mathcal{P}(\mathcal{B})$ if and only if the functor $G$ is exact and faithful;
    \item  Assume that $F$ is exact. Then $F$ is faithful if and only if the unit $\eta\colon {\rm Id}_\mathcal{A}\rightarrow GF$ is mono.
\end{enumerate}
\end{lem}

\begin{proof}
(1) For the ``if" part, we assume that $G$ is exact. The adjoint pair implies a natural isomorphism
\begin{align}\label{equ:adj}
{\rm Hom}_\mathcal{B}(F(P), -)\stackrel{\sim}{\longrightarrow} {\rm Hom}_\mathcal{A}(P, G(-))
\end{align}
for object $P$ in $\mathcal{A}$. If $P$ is projective, the functor on  the right hand side is exact. It follows that $F(P)$ is projective.

For the ``only if" part, we recall that the functor $G$, as a right adjoint,  is automatically left exact. Let $u\colon X\rightarrow Y$ be an epimorphism in $\mathcal{B}$. We claim that $G(u)$ is also epic. Then we are done.

Since $\mathcal{A}$ has enough projective objects, the claim is equivalent to the statement that ${\rm Hom}_\mathcal{A}(P, G(u))$ is surjective for each $P\in \mathcal{P}(A)$. By assumption, $F(P)$ is projective and thus ${\rm Hom}_\mathcal{B}(F(P), u)$ is surjective. Then the claim follows from the natural isomorphism (\ref{equ:adj}).

(2) For the ``if" part, we already have $F(\mathcal{P}(\mathcal{A}))\subseteq \mathcal{P}(\mathcal{B})$ by (1). Take any object $Q\in \mathcal{B}$ and consider the counit $\varepsilon_Q\colon FG(Q)\rightarrow Q$. We observe that $G(\varepsilon_Q)$ is epic by the identity
$$G(\varepsilon_Q)\circ \eta_{G(Q)}={\rm Id}_{G(Q)}.$$
 Since $G$ is exact and faithful, it detects epimorphisms. In particular, we infer that $\varepsilon_Q$ is epic.

Assume now that $Q$ is projective in $\mathcal{B}$. Take an epimorphism $v\colon P\rightarrow G(Q)$ with $P\in \mathcal{P}(A)$. Then $F(v)$ is epic by the right-exactness of $F$. The epimorphism $\varepsilon_Q\circ F(v)\colon F(P)\rightarrow Q$ splits. Consequently, $Q$ is isomorphic to a direct summand of $F(P)$, as required.

For the ``only if" part, by (1) it suffices to show the faithfulness of $G$. We take a nonzero morphism $w$ in $\mathcal{B}$. Since $\mathcal{B}$ has enough projective objects, there exists a projective object $Q$ satisfying ${\rm Hom}_\mathcal{B}(Q, w)\neq 0$. By the assumption, we may assume that $Q$ is a direct summand of $F(P)$ for some projective object $P$ in $\mathcal{A}$. Consequently, ${\rm Hom}_\mathcal{B}(F(P), w)\neq 0$. By the natural isomorphism (\ref{equ:adj}), we infer that ${\rm Hom}_\mathcal{A}(P, G(w))\neq 0$. In particular, the morphism $G(w)$ is nonzero.

(3) The ``only if" part follows by a dual argument in the first paragraph  of the proof of (2).

 For the ``if" part, we take a morphism $a\colon X\rightarrow Y$ in $\mathcal{A}$ satisfying $F(a)=0$. The naturalness of $\eta$ yields
$$\eta_Y\circ a=GF(a)\circ \eta_X.$$
It follows that $\eta_Y\circ a=0$. Since $\eta_Y$ is mono, it forces that $a=0$, proving the faithfulness of $F$.
\end{proof}

The notions of a Frobenius pair and a Frobenius functor are very classical; see \cite{Mor65, CGN99, CDM02}.

Let $F\colon \mathcal{C}\rightarrow \mathcal{D}$ and $G\colon \mathcal{D}\rightarrow \mathcal{C}$ be two additive functors. Assume that $\alpha\colon \mathcal{C}\rightarrow \mathcal{C}$ and $\beta\colon \mathcal{D}\rightarrow \mathcal{D}$ are two autoequivalences.

 We say that $(F, G)$ is a \emph{Frobenius pair} of type $(\alpha, \beta)$ between $\mathcal{C}$ and $\mathcal{D}$,  provided that both $(F, G)$ and $(G, \beta F\alpha)$ are adjoint pairs. We call the functor $F$ a \emph{Frobenius functor}, if it fits into a Frobenius pair $(F, G)$.  We observe that $G$ is also a Frobenius functor. In other words, Frobenius functors always appear in pairs. We mention that a  Frobenius pair of type $(\alpha, \beta)$ is called a $(\alpha, \beta^{-1})$-strongly adjoint pair in \cite{Mor65}.

By a \emph{classical Frobenius pair} $(F, G)$, we mean a Frobenius pair of type $({\rm Id}_\mathcal{C}, {\rm Id}_\mathcal{D})$. That is, both $(F, G)$ and $(G, F)$ are adjoint pairs.

The following observations on Frobenius pairs between abelian categories are well known.

\begin{cor}\label{cor:fr}
Let $(F, G)$ be a Frobenius pair of type $(\alpha, \beta)$ between $\mathcal{A}$ and $\mathcal{B}$. Then the following statements hold.
\begin{enumerate}
\item Both $F$ and $G$ are exact satisfying $F(\mathcal{P}(\mathcal{A}))\subseteq \mathcal{P}(\mathcal{B})$ and $G(\mathcal{P}(\mathcal{B}))\subseteq \mathcal{P}(\mathcal{A})$.
\item The functor $F$ is faithful if and only if ${\rm add}\; G(\mathcal{P}(\mathcal{B}))=\mathcal{P}(\mathcal{A})$, if and only if the unit $\eta\colon {\rm Id}_\mathcal{A}\rightarrow GF$ is mono.
    \item The functor $G$ is faithful if and only if ${\rm add}\; F(\mathcal{P}(\mathcal{A}))=\mathcal{P}(\mathcal{B})$, if and only if the counit $\varepsilon \colon FG\rightarrow {\rm Id}_\mathcal{B}$ is epic.
\end{enumerate}
\end{cor}

\begin{proof} We observe that $F$ is faithful (\emph{resp}. left exact) if and only if so is $\beta F\alpha$. Then we apply Lemma \ref{lem:adj} to the adjoint pairs $(F, G)$ and $(G, \beta F\alpha)$.
\end{proof}

The following examples are our main concerns. For an arbitrary ring $R$, we denote by $R\mbox{-Mod}$ the abelian category of left $R$-modules. By ${\rm Hom}_R(-, -)$ we mean the Hom bifunctor on $R\mbox{-Mod}$. As usual, right $R$-modules are viewed as left $R^{\rm op}$-modules. For Frobenius extensions and Frobenius bimodules, we refer to \cite{Kad}.

\begin{exm}\label{exm:Fex}
{\rm Let $G$ be a group and $H\subseteq G$ be a subgroup of finite index. Let $R$ be a ring. Denote by $RG$ and $RH$ the group rings. We have the restriction functor ${\rm Res}\colon RG\mbox{-Mod}\rightarrow RH\mbox{-Mod}$, the induction functor ${\rm Ind}=RG\otimes_{RH}-$ and coinduction functor ${\rm Coind}={\rm Hom}_{RH}(RG, -)$ from $RH\mbox{-Mod}$ to $RG\mbox{-Mod}$. They form well-known adjoint pairs $({\rm Ind}, {\rm Res})$ and $({\rm Res}, {\rm Coind})$. The functors ${\rm Ind}$ and ${\rm Coind}$ are isomorphic. Then we have a classical Frobenius pair $({\rm Ind}, {\rm Res})$.

Indeed, the above example generalizes as follows. For a ring extension $\theta\colon R\rightarrow S$, we have the scalar-extension functor $S\otimes_R-\colon R\mbox{-Mod}\rightarrow S\mbox{-Mod}$ and the forgetful functor $U\colon S\mbox{-Mod}\rightarrow R\mbox{-Mod}$. Recall that  $\theta$ is a \emph{Frobenius extension},  provided that $S$ is finitely generated projective as a left $R$-module and that there is an isomorphism
$$S\simeq {\rm Hom}_R(S, R)$$
  of $S$-$R$-bimodules. We observe that the ring extension $\theta\colon R\rightarrow S$ is Frobenius if and only if  $(S\otimes_R-, U)$ is a classical Frobenius pair.

  For another well-known example, let $t\geq 2$ and $S=R[x]/(x^t)$ be the truncated polynomial extension. Then the natural embedding $\theta\colon R\rightarrow S$ is a Frobenius extension.}
\end{exm}

 Frobenius functors between module categories are determined by Frobenius bimodules. The following example generalizes the above one.

 \begin{exm}
 {\rm Let $S$ and $R$ be two rings. An $S$-$R$-bimodule $M={_SM_R}$ is called a \emph{Frobenius bimodule}, provided that both $_SM$ and $M_R$ are finitely generated projective such that there is an isomorphism of $R$-$S$-bimodules
 $${\rm Hom}_{S}(M, S)\simeq {\rm Hom}_{R^{\rm op}}(M, R).$$
 Denote the common $R$-$S$-bimodule by  $N$. Then $(M\otimes_R-, N\otimes_S-)$ is a classical Frobenius pair between $R\mbox{-Mod}$ and $S\mbox{-Mod}$. Indeed, any classical Frobenius pair between these two module categories is of this form; see \cite[Theorem~2.1]{CGN99}. We observe that a ring extension $\theta\colon R\rightarrow S$ is a Frobenius extension if and only if the natural $S$-$R$-bimodule ${_SS_R}$ is Frobenius.

 We mention that Frobenius bimodules arise naturally in stable equivalences of Morita type; see \cite{DM07, Xi08}.}
 \end{exm}

The following non-classical Frobenius pair plays a fundamental role in modern differential graded theory.

\begin{exm}\label{exm:dg}
{\rm Let $A=(\oplus_{p\in \mathbb{Z}} A^p, d_A)$ be a differential graded ring and $A^\sharp$ be its underlying $\mathbb{Z}$-graded ring. Recall that a left differential graded $A$-module $M=(\oplus_{p\in \mathbb{Z}}M^p, d_M)$ consists of a left $\mathbb{Z}$-graded $A$-module and a differential $d_M\colon M\rightarrow M$ of degree one subject to the following graded Leibniz rule
$$d_M(a.m)=d_A(a).m +(-1)^{p} a.d_M(m)$$
for each $a\in A^p$ and $m\in M$.

Denote by $A\mbox{-DGMod}$ the abelian category of left differential graded $A$-modules, where the morphisms respect the grading and differentials. Denote by $\Sigma(M)$ the translated differential graded module, which is given by $\Sigma(M)^p=M^{p+1}$, $d_{\Sigma(M)}=-d_M$ and $a_\circ m=(-1)^pa.m$ for $a\in A^p$; here, ``$_\circ$" denotes the $A$-action on $\Sigma(M)$ and ``$.$" the original $A$-action on $M$. Then we have the  translation functor $\Sigma$ on $A\mbox{-DGMod}$, which is an automorphism.

Denote by $A^\sharp\mbox{-Gr}$ the abelian category of left $\mathbb{Z}$-graded $A^\sharp$-modules, whose morphisms respect the grading. Then we have the forgetful functor
$$U\colon A\mbox{-DGMod}\longrightarrow A^\sharp\mbox{-Gr}.$$

Each $\mathbb{Z}$-graded $A^\sharp$-module $X=\oplus_{p\in \mathbb{Z}} X^p$ gives rise to a differential graded $A$-module $F(X)$ as follows: \begin{center}
$F(X)^p=X^{p}\oplus X^{p-1}$, $d_{F(X)}\binom{x}{y}=\binom{0}{x}$ and $a.\binom{x}{y}=\binom{(-1)^qa.x+d_A(a).y}{a.y}$
\end{center}
for $a\in A^q$, $x\in X^{p}$ and $y\in X^{p-1}$. This defines a functor
$$F\colon A^\sharp\mbox{-Gr}\longrightarrow A\mbox{-DGMod}.$$
By \cite[Subsection~2.2]{Kel} we have adjoint pairs $(F, U)$ and $(U, \Sigma F)$. In other words, we have a Frobenius pair $(F,U)$ of type $({\rm Id}_{A^\sharp\mbox{-}{\rm Gr}}, \Sigma)$ between $A^\sharp\mbox{-Gr}$ and $A\mbox{-DGMod}$.

To be more specific, we assume that $R$ is an ordinary ring, which is viewed as a differential graded ring concentrated in degree zero. Then $R\mbox{-DGMod}$ coincides with the category $C(R\mbox{-Mod})$ of cochain complexes of $R$-modules, and $R^\sharp\mbox{-Gr}$ is isomorphic to the product $\prod_{\mathbb{Z}} R\mbox{-Mod}$ of countable copies of $R\mbox{-Mod}$. Consequently, we have a Frobenius pair $(F, U)$ between $C(R\mbox{-Mod})$ and $\prod_{\mathbb{Z}} R\mbox{-Mod}$.
}\end{exm}

\section{Frobenius pairs and Gorenstein projective objects}

In this section, we prove that a faithful Frobenius functor preserves the Gorenstein projective dimension of objects; see Theorem~\ref{thm:GP}. Consequently, it preserves and reflects Gorenstein projective objects.

Let $\mathcal{A}$ be an abelian category with enough projective objects. Recall that an acyclic complex $$P^\bullet = \cdots\longrightarrow P^{n+1}\longrightarrow P^{n}\longrightarrow P^{n-1}\longrightarrow\cdots$$
 of projective objects is said to be \emph{totally acyclic}, provided it remains acyclic after applying ${\rm Hom}_{\mathcal{A}}(-, Q)$ for any projective object $Q\in \mathcal{A}$. An object $M\in \mathcal{A}$ is \emph{Gorenstein projective} \cite{EJ00, Hol04} if there is  a totally acyclic complex $P^\bullet$  such that $M$ is isomorphic to its zeroth cocycle $Z^0(P^\bullet)$. Denote by $\mathcal{GP}(\mathcal{A})$ the full subcategory of $\mathcal{A}$ consisting of Gorenstein projective objects. We observe $\mathcal{P}(\mathcal{A})\subseteq \mathcal{GP}(\mathcal{A})$.

For an object $X$ in $\mathcal{A}$,  we denote by ${\rm pd}_\mathcal{A}(X)$ its projective dimension. The \emph{Gorenstein projective dimension} ${\rm Gpd}_\mathcal{A}(X)$ of $X$ is defined such that ${\rm Gpd}_\mathcal{A}(X)\leq n$ if and only if there is an exact sequence $0\rightarrow M^{-n}\rightarrow M^{-n+1} \rightarrow \cdots \rightarrow M^{0}\rightarrow X\rightarrow 0$ with each $M^{-i}\in \mathcal{GP}(\mathcal{A})$. By definition, ${\rm Gpd}_\mathcal{A}(X)=0$ if and only if $X\in \mathcal{GP}(\mathcal{A})$.

The following facts are standard. Denote by ${^\perp \mathcal{P}(\mathcal{A})}$ the full subcategory of $\mathcal{A}$ formed by those objects $Y$ satisfying ${\rm Ext}^i_\mathcal{A}(Y, P)=0$ for any $i\geq 1$ and $P\in \mathcal{P}(\mathcal{A})$.

\begin{lem}\label{lem:GP}
Keep the notation as above. Then the following statements hold.
\begin{enumerate}
\item The full subcategory $\mathcal{GP}(\mathcal{A})$ of $\mathcal{A}$ is closed under extensions, kernels of epimorphisms,  and direct summands.
\item An object $M$ is Gorenstein projective if and only if $M\in {^\perp \mathcal{P}(\mathcal{A})}$ and there is an exact sequence $0\rightarrow M \rightarrow P^0\rightarrow P^1 \rightarrow \cdots$ with each $P^i\in \mathcal{P}(\mathcal{A})$ and  each cocycle in ${^\perp\mathcal{P}(A)}$.
\item Assume that ${\rm Gpd}_\mathcal{A}(X)\leq n$. Then for each exact sequence $0\rightarrow K\rightarrow P^{1-n} \rightarrow \cdots \rightarrow P^{-1} \rightarrow P^0\rightarrow X\rightarrow 0$ with $P^{-i}\in \mathcal{P}(\mathcal{A})$, we have $K\in \mathcal{GP}(\mathcal{A})$.
\item  Assume that ${\rm pd}_\mathcal{A}(X)$ is finite. Then we have ${\rm Gpd}_\mathcal{A}(X)={\rm pd}_\mathcal{A}(X)$.
\end{enumerate}
\end{lem}

\begin{proof}
For (1), we refer to \cite[Theorem~2.5]{Hol04}, and for (2), we refer to \cite[Proposition~2.3]{Hol04}. For (3), we refer to \cite[Theorem~2.20]{Hol04}, and for (4), we refer to \cite[Proposition~2.27]{Hol04}.
\end{proof}

The main result of this paper is as follows. It strengthens and extends \cite[Proposition~4.5]{LX}.

\begin{thm}\label{thm:GP}
Let $F\colon \mathcal{A}\rightarrow \mathcal{B}$ be a faithful Frobenius functor between two abelian categories with enough projective objects. Then we have
$${\rm Gpd}_\mathcal{A}(X)={\rm Gpd}_\mathcal{B}(F(X))$$
 for each object $X\in \mathcal{A}$. In particular, $X\in \mathcal{GP}(\mathcal{A})$ if and only if  $F(X)\in \mathcal{GP}(\mathcal{B})$.
\end{thm}

\begin{proof}
By Corollary~\ref{cor:fr}, the result follows directly from  Proposition~\ref{prop:GP} below.
\end{proof}

\begin{rem}
By categorical duality, there is a similar result for Gorenstein injective dimensions. Here, we assume that both $\mathcal{A}$ and $\mathcal{B}$ have enough injective objects. Denote by ${\rm Gid}$ the Gorenstein injective dimension. Then a faithful Frobenius functor $F\colon \mathcal{A}\rightarrow \mathcal{B}$ satisfies $${\rm Gid}_\mathcal{A}(X)={\rm Gid}_\mathcal{B}(F(X))$$
 for any object $X$ in $\mathcal{A}$.
\end{rem}

Recall that the \emph{global Gorenstein dimension} of $\mathcal{A}$ is defined as
$${\rm gl.Gdim}(\mathcal{A})={\rm sup}\{{\rm Gpd}_\mathcal{A}(X)\; |\; X\in \mathcal{A}\}.$$
It might be also called the global Gorentein projective dimension. The preference of the terminology is justified as follows: if $\mathcal{A}$ also has enough injective objects,  ${\rm gl.Gdim}(\mathcal{A})$ equals the supremum of the Gorenstein injective dimension of all objects in $\mathcal{A}$. We refer the details to \cite[Chapter~VII]{BR}; see also Appendix A.

\begin{cor}\label{cor:Ggd}
Let $(F, G)$ be a Frobenius pair between $\mathcal{A}$ and $\mathcal{B}$. Assume that both $F$ and $G$ are faithful. Then ${\rm gl.Gdim}(\mathcal{A})={\rm gl.Gdim}(\mathcal{B})$.
\end{cor}

\begin{proof}
By Theorem~\ref{thm:GP} we have  ${\rm gl.Gdim}(\mathcal{A})\leq \rm {gl.Gdim}(\mathcal{B})$. Since $G\colon \mathcal{B}\rightarrow \mathcal{A}$ is also a faithful Frobenius functor, we have ${\rm gl.Gdim}(\mathcal{B})\leq {\rm gl.Gdim}(\mathcal{A})$.
\end{proof}

\begin{rem}
We mention that the above faithful condition is necessary. Take another abelian category $\mathcal{B}'$ with enough projective objects. An object $(B, B')$ in the product category $\mathcal{B}\times \mathcal{B}'$ is Gorenstein projective if and only if so are both $B$ and $B'$.

Consider the canonical projection ${\rm Pr}\colon \mathcal{B}\times \mathcal{B}'\rightarrow \mathcal{B}$ and the inclusion ${\rm Inc}\colon \mathcal{B}\rightarrow \mathcal{B}\times \mathcal{B}'$. Then we have a classical Frobenius pair $({\rm Pr}, {\rm Inc})$. However, the functor ${\rm Pr}$ does not reflect Gorenstein projective objects in general. We observe that the global Goresntein dimension of $\mathcal{B}\times \mathcal{B}'$ is, in general, larger than the one of $\mathcal{B}$.
\end{rem}

\begin{lem}\label{lem:geq}
Let $F\colon \mathcal{A}\rightleftarrows \mathcal{B}\colon G$ be an adjoint pair consisting of exact functors. Assume that $G(\mathcal{P}(\mathcal{B}))\subseteq \mathcal{P}(\mathcal{A})$. Then we have ${\rm Gpd}_\mathcal{A}(X)\geq {\rm Gpd}_\mathcal{B}(F(X))$ for each object $X\in \mathcal{A}$.
\end{lem}

\begin{proof}
Since $F$ is exact, it suffices to claim that $F(X)\in \mathcal{GP}(\mathcal{B})$ for any $X\in \mathcal{GP}(\mathcal{A})$. Take a totally acyclic complex $P^\bullet$ with $X\simeq Z^0(P^\bullet)$.
 By Lemma \ref{lem:adj}(1), the functor $F$ preserves projective objects. Hence, the acyclic complex $F(P^\bullet)$ consists of projective objects. For each projective object $Q\in \mathcal{B}$, we have an isomorphism
 $${\rm Hom}_\mathcal{B}(F(P^\bullet), Q)\stackrel{\sim}\longrightarrow  {\rm Hom}_\mathcal{A}(P^\bullet, G(Q))$$ of complexes. By assumption $G(Q)$ is projective in $\mathcal{A}$. It follows that these two Hom complexes are acyclic. In particular, the complex $F(P^\bullet)$ is totally acyclic. Since $F(X)\simeq Z^0(F(P^\bullet))$, we are done with the claim.
\end{proof}

The following result is slightly stronger than Theorem \ref{thm:GP}, although its assumptions seem to be technical.

\begin{prop}\label{prop:GP}
Let $F\colon \mathcal{A}\rightleftarrows \mathcal{B}\colon G$ be an adjoint pair consisting of exact functors. Assume that $F$ is faithful and that ${\rm add}\; G(\mathcal{P}(\mathcal{B}))=\mathcal{P}(\mathcal{A})$.
Then we have ${\rm Gpd}_\mathcal{A}(X)= {\rm Gpd}_\mathcal{B}(F(X))$ for each object $X\in \mathcal{A}$.
\end{prop}

\begin{proof}
By Lemma \ref{lem:geq}, it suffices to show that ${\rm Gpd}_\mathcal{A}(X) \leq {\rm Gpd}_\mathcal{B}(F(X))$.  We first claim that if $F(X)\in \mathcal{GP}(\mathcal{B})$, then $X\in \mathcal{GP}(\mathcal{A})$.

For the claim, we take exact sequence
$$0\longrightarrow F(X)\stackrel{f'}\longrightarrow P\longrightarrow Y\longrightarrow 0$$ in $\mathcal{B}$  with $P\in \mathcal{P}(\mathcal{B})$ and $Y\in \mathcal{GP}(\mathcal{B})$. Denote by $\eta$ and $\varepsilon$ the unit and counit of the adjoint pair $(F, G)$, respectively. Set $f=G(f')\circ \eta_{X}\colon X\rightarrow G(P)$. By $f'= \varepsilon_{P}\circ F(f)$, we have that $F(f)$ is a monomorphism. Since $F$ is exact and faithful, we infer that $f$ is also a monomorphism. Consequently, we obtain an exact sequence in $\mathcal{A}$
$$0\longrightarrow X\stackrel{f}\longrightarrow G(P)\longrightarrow X'\longrightarrow 0.$$

The following commutative exact diagram
$$\xymatrix{
0\ar[r] &F(X)\ar[r]^{F(f)}\ar@{=}[d] & FG(P)\ar[r]\ar[d]^{\varepsilon_{P}} & F(X')\ar[r]\ar[d]^{} &0\\
0\ar[r] &F(X) \ar[r]^{f'} &P \ar[r] &Y\ar[r] & 0
}$$
yields an exact sequence
$$0\longrightarrow FG(P)\longrightarrow F(X')\oplus P \longrightarrow Y\longrightarrow 0.$$
By Lemma~\ref{lem:adj}(1), the functor $F$ preserves projective objects. It follows that $FG(P)$ is projective. Since $Y$ is Gorenstein projective, it follows from the above exact sequence and Lemma~\ref{lem:GP}(1) that $F(X')$ is Gorenstein projective. Hence, by Lemma~\ref{lem:GP}(2), we have ${\rm Ext}_{\mathcal{B}}^{i}(F(X'), Q)=0$ for any $Q\in \mathcal{P}(\mathcal{B})$ and $i> 0$. However, by the adjoint pair $(F, G)$, we infer that  ${\rm Ext}_{\mathcal{A}}^{i}(X', G(Q))=0$. By the assumption  that ${\rm add}\; G(\mathcal{P}(\mathcal{B}))=\mathcal{P}(\mathcal{A})$, we conclude that ${\rm Ext}_\mathcal{A}^{i}(X', P)=0$ for any $i>0$ and $P\in \mathcal{P}(\mathcal{A})$, that is, $X'\in {^{\perp}}\mathcal{P}(\mathcal{A})$.

Set $Q^0=G(P)$ and repeat the above argument for $X'$. We will obtain inductively  an exact sequence
$$0\longrightarrow X\longrightarrow Q^{0}\longrightarrow Q^{1}\longrightarrow Q^{2}\longrightarrow \cdots $$
with each $Q^i\in \mathcal{P}(\mathcal{A})$ and cocycle in ${^{\perp}}\mathcal{P}(\mathcal{A})$. In view of Lemma \ref{lem:GP}(2), we deduce the claim.

For the required inequality, we may assume that ${\rm Gpd}_\mathcal{B}(F(X))=n<\infty$. Take an exact sequence
$$\xi\colon 0\longrightarrow K\longrightarrow P^{1-n} \longrightarrow \cdots \longrightarrow P^0\longrightarrow X\longrightarrow 0$$
 with each $P^i\in \mathcal{P}(\mathcal{A})$. Applying the exact functor $F$ to $\xi$ and Lemma \ref{lem:GP}(3), we infer that $F(K)$ is Gorenstein projective. The above claim for $K$ implies that $K$ is Gorenstein projective, proving that ${\rm Gpd}_\mathcal{A}(X)\leq n$.
\end{proof}

We apply Theorem \ref{thm:GP} to recover a number of   known results.

\begin{exm}\label{exm:Fxe2}
{\rm Recall from Example~\ref{exm:Fex} that a Frobenius extension $\theta\colon R\rightarrow S$ yields a classical Frobenius pair $(S\otimes_R-, U)$ between $S\mbox{-Mod}$ and $R\mbox{-Mod}$. Applying Theorem~\ref{thm:GP} to the forgetful functor $U$, we infer that  a left $S$-module $_SM$ is Gorenstein projective if and only if the underlying $R$-module $_RM$ is Gorenstein projective.  This result is due to \cite[Theorem~2.2]{Ren2} and \cite[Theorem~3.2]{Zhao}.

In the Frobenius extension $\theta\colon R\rightarrow S$, we assume further that $S$ is a generator as a right $R$-module. Then $S\otimes_R-\colon R\mbox{-Mod}\rightarrow S\mbox{-Mod}$ is faithful. We apply Corollary~\ref{cor:Ggd} to obtain that $S$ is left $n$-Gorenstein if and only if so is $R$. This generalizes \cite[Theorem~3.3]{Ren2}.  Here, we use the fact that a ring $R$ is left $n$-Gorenstein if and only if ${\rm gl.Gdim}(R\mbox{-Mod})\leq n$.

 More concretely, let us take $S=RG$ for a finite group $G$. We infer the following classical result: an $RG$-module $_{RG}M$ is Gorenstein projective if and only if the underlying $R$-module $_RM$ is Gorenstein projective;  moreover, the group ring $RG$ is left $n$-Gorenstein if and only if so is $R$.

Let us mention that applying Theorem~\ref{thm:GP} to the truncated polynomial extension $R\rightarrow R[x]/(x^t)$, we might recover some results in \cite{Wei, XYY}; compare \cite[Section~3]{Ren1} and \cite{TH}.}

\end{exm}

\begin{exm}
{\rm Let $R$ be a ring. Consider the Frobenius pair $(F, U)$ between $C(R\mbox{-Mod})$ and $\prod_{\mathbb{Z}} R\mbox{-Mod}$ in Example~\ref{exm:dg}. We observe that a Gorenstein projective object in the product category $\prod_{\mathbb{Z}} R\mbox{-Mod}$ is precisely given by a sequence $(G^n)_{n\in \mathbb{Z}}$ of Gorenstein projective $R$-modules.

Applying Theorem~\ref{thm:GP} to the forgetful functor $U$, we infer that a complex $X^\bullet=(X^n, d_X^n)_{n\in \mathbb{Z}}$ is a Gorenstein projective object in $C(R\mbox{-Mod})$ if and only if each component $X^n$ is a Gorenstein projective $R$-module. This result is due to \cite[Theorem~1]{Yang11} and \cite[Theorem~2.2]{YL11}; see also \cite[Corollary~3.3]{Ren1}.
}\end{exm}

\section{Triangle equivalences induced by Frobenius functors}

In this section, we study when a Frobenius functor induces triangle equivalences on the stable categories of Gorenstein projective objects, the singularity categories and the Gorenstein defect categories, respectively.

Let $\mathcal{A}$ be an abelian category with enough projective objects. As $\mathcal{GP}(\mathcal{A})$ is closed under extensions, it is naturally an exact category in the sense of Quillen. Moreover, it is a Frobenius category, whose projective-injective objects are precisely projective objects in $\mathcal{A}$. Therefore, by the general result in \cite[I.2]{Hap88},  its stable category $\underline{\mathcal{GP}}(\mathcal{A})$ is  naturally a triangulated category.

Denote by $\mathbf{D}^b(\mathcal{A})$ the bounded derived category.   The bounded homotopy category $\mathbf{K}^b(\mathcal{P}(A))$ is naturally viewed as a triangulated subcategory of  $\mathbf{D}^b(\mathcal{A})$. We denote by $\mathbf{D}^b(\mathcal{A})_{\rm fGd}$ the full subcategory of $\mathbf{D}^b(\mathcal{A})$ formed by those complexes isomorphic to a bounded complex of Gorenstein projective objects. Then we have $\mathbf{K}^b(\mathcal{P}(A))\subseteq \mathbf{D}^b(\mathcal{A})_{\rm fGd}$.

Following \cite{Buc, Or04}, the \emph{singularity category} of $\mathcal{A}$ is defined to be the following Verdier quotient triangulated category
$$\mathbf{D}_{\rm sg}(\mathcal{A})= \mathbf{D}^b(\mathcal{A})/{\mathbf{K}^b(\mathcal{P}(A))}.$$
We observe that $\mathbf{D}_{\rm sg}(\mathcal{A})$ vanishes if and only if each object in $\mathcal{A}$ has finite projective dimension.

There is a canonical functor
$${\rm can}\colon \underline{\mathcal{GP}}(\mathcal{A})\longrightarrow \mathbf{D}_{\rm sg}(\mathcal{A})$$
sending a Gorenstein projective object to the corresponding stalk complex concentrated in degree zero. The functor is well defined, since projective objects vanish in  $\mathbf{D}_{\rm sg}(\mathcal{A})$. This canonical functor is a triangle functor in a natural way; see \cite[Lemma~2.5]{Chen11}.

The following fundamental result is due to \cite[Theorem~4.4.1]{Buc}; compare \cite[Corollary~4.13]{Bel00} and \cite[Proposition~1.21]{Or04}.

\begin{lem}\label{lem:buc}
The above canonical functor is fully faithful, and it induces a triangle equivalence
$$\underline{\mathcal{GP}}(\mathcal{A})\simeq {\mathbf{D}^b(\mathcal{A})_{\rm fGd}}/{\mathbf{K}^b(\mathcal{P}(A))}.$$
\end{lem}

\begin{proof}
For a detailed proof of the fully-faithfulness, we refer to \cite[Theorem~2.1]{Chen11}. The smallest triangulated subcategory of $\mathbf{D}_{\rm sg}(\mathcal{A})$ containing all Gorenstein projective objects is clearly ${\mathbf{D}^b(\mathcal{A})_{\rm fGd}}/{\mathbf{K}^b(\mathcal{P}(A))}$. Thus, the denseness follows immediately.
\end{proof}

Following \cite{BJO}, the \emph{Gorenstein defect category} of $\mathcal{A}$ is defined to be
$$\mathbf{D}_{\rm def}(\mathcal{A})=\mathbf{D}^b(\mathcal{A})/{\mathbf{D}^b(\mathcal{A})_{\rm fGd}}.$$
The terminology is justified by the following fact: $\mathbf{D}_{\rm def}(\mathcal{A})$ vanishes if and only if each object in $\mathcal{A}$ has finite Gorenstein projective dimension. In view of Lemma~\ref{lem:buc}, we have a short exact sequence of triangulated categories
$$ \underline{\mathcal{GP}}(\mathcal{A})\stackrel{\rm can}\longrightarrow \mathbf{D}_{\rm sg}(\mathcal{A})\longrightarrow \mathbf{D}_{\rm def}(\mathcal{A}),$$
where the unnamed arrow is the quotient functor.

Assume that  $F\colon \mathcal{A}\rightleftarrows \mathcal{B}\colon G$  is a Frobenius pair of type $(\alpha, \beta)$. Here, $\alpha$ and $\beta$ are autoequivalences on $\mathcal{A}$ and $\mathcal{B}$, respectively. The autoequivalence $\alpha$ induces triangle autoequivalences on $ \underline{\mathcal{GP}}(\mathcal{A})$, $\mathbf{D}_{\rm sg}(\mathcal{A})$ and $\mathbf{D}_{\rm def}(\mathcal{A})$, respectively. All the induced autoequivalences will still be denoted by $\alpha$. Similar remarks apply to $\beta$.

By Corollary~\ref{cor:fr}, both the exact functors $F$ and $G$ respect projective objects. By Lemma~\ref{lem:geq}, both functors respect Gorenstein projective objects. Consequently, we have the following commutative diagram
\[\xymatrix{
\underline{\mathcal{GP}}(\mathcal{A}) \ar@<-.7ex>[d]_-{F} \ar[rr]^-{\rm can} && \mathbf{D}_{\rm sg}(\mathcal{A}) \ar@<-.7ex>[d]_-{F} \ar[rr] && \mathbf{D}_{\rm def}(\mathcal{A}) \ar@<-.7ex>[d]_-{F}\\
\underline{\mathcal{GP}}(\mathcal{B}) \ar@<-.7ex>[u]_-{G} \ar[rr]^{\rm can} && \mathbf{D}_{\rm sg}(\mathcal{B}) \ar@<-.7ex>[u]_-{G} \ar[rr] && \mathbf{D}_{\rm def}(\mathcal{B})\ar@<-.7ex>[u]_-{G}
}\]
where the vertical induced functors form three adjoint pairs of triangle functors. They are all Frobenius pairs of type $(\alpha, \beta)$. Here, we use the same letter to denote the induced triangle functors.

The following results study when the induced adjoint pairs give rise to equivalences. Recall that  $\eta\colon {\rm Id}_\mathcal{A}\rightarrow GF$ and $\varepsilon\colon FG\rightarrow {\rm Id}_\mathcal{B}$ are the unit and counit of the adjoint pair $(F, G)$, respectively.

\begin{prop}\label{prop:tri-equ}
Let  $F\colon \mathcal{A}\rightleftarrows \mathcal{B}\colon G$ be  a Frobenius pair with both $F$ and $G$ faithful. Then the following statements hold.
\begin{enumerate}
\item The induced adjoint pair $F\colon \underline{\mathcal{GP}}(\mathcal{A})\rightleftarrows \underline{\mathcal{GP}}(\mathcal{B})\colon G$ yields mutually inverse equivalences if and only if ${\rm Cok}(\eta_X)$ has projective dimension at most one and ${\rm Ker}(\varepsilon_Y)$ is projective for any $X\in \mathcal{GP}(\mathcal{A})$ and $Y\in \mathcal{GP}(\mathcal{B})$;
\item  The induced adjoint pair $F\colon \mathbf{D}_{\rm sg}(\mathcal{A})\rightleftarrows \mathbf{D}_{\rm sg}(\mathcal{B})\colon G$ yields mutually inverse equivalences if and only if both ${\rm Cok}(\eta_X)$ and ${\rm Ker}(\varepsilon_Y)$ have finite projective dimension for any $X\in \mathcal{A}$ and $Y\in \mathcal{B}$;
\item  The induced adjoint pair $F\colon \mathbf{D}_{\rm def}(\mathcal{A})\rightleftarrows \mathbf{D}_{\rm def}(\mathcal{B})\colon G$ yields mutually inverse equivalences if and only if both ${\rm Cok}(\eta_X)$ and ${\rm Ker}(\varepsilon_Y)$ have finite Gorenstein projective dimension for any $X\in \mathcal{A}$ and $Y\in \mathcal{B}$.
\end{enumerate}
\end{prop}

\begin{proof}
By the faithfulness assumption and Corollary~\ref{cor:fr}, the unit $\eta$ is mono and the counit $\varepsilon$ is epic.

(1) For the statement, it suffices to study when precisely $\eta_X\colon X\rightarrow GF(X)$ and $\varepsilon_Y\colon FG(Y)\rightarrow Y$ become isomorphisms in the stable categories for any $X\in \mathcal{GP}(\mathcal{A})$ and $Y\in \mathcal{GP}(\mathcal{B})$.

 As $\eta_X$ is mono, we have a short exact sequence in $\mathcal{A}$
$$0\longrightarrow X\stackrel{\eta_X} \longrightarrow GF(X)\longrightarrow {\rm Cok}(\eta_X)\longrightarrow 0.$$
It induces an exact triangle in $\mathbf{D}_{\rm sg}(\mathcal{A})$.
$$X\stackrel{\eta_X} \longrightarrow GF(X)\longrightarrow {\rm Cok}(\eta_X)\longrightarrow \Sigma(X)$$
where $\Sigma$ is the suspension functor. Therefore, $\eta_X$ is an isomorphism in $\underline{\mathcal{GP}}(\mathcal{A})$, or equivalently via the canonical functor ``can", $\eta_X$ is an isomorphism in $\mathbf{D}_{\rm sg}(\mathcal{A})$ if and only if ${\rm Cok}(\eta_X)$ is zero in $\mathbf{D}_{\rm sg}(\mathcal{A})$. The latter condition means exactly that ${\rm Cok}(\eta_X)$ has finite projective dimension. As both $X$ and $FG(X)$ are Gorenstein projective, we have ${\rm Gpd}_\mathcal{A}({\rm Cok}(\eta_X))\leq 1$. By Lemma~\ref{lem:GP}(4), we infer ${\rm pd}_\mathcal{A}({\rm Cok}(\eta_X))\leq 1$. In summary, we have proved that $\eta_X$ is an isomorphism in $\underline{\mathcal{GP}}(\mathcal{A})$ if and only if ${\rm pd}_\mathcal{A}({\rm Cok}(\eta_X))\leq 1$.

Since $\varepsilon_Y$ is epic, we infer by Lemma~\ref{lem:GP}(1) that ${\rm Ker}(\varepsilon_Y)$ is Gorenstein projective, that is, ${\rm Gpd}_\mathcal{B}({\rm Ker}(\varepsilon_Y))=0$. Then the same argument as above will prove  that $\varepsilon_Y$ is an isomorphism in $\underline{\mathcal{GP}}(\mathcal{B})$ if and only if ${\rm pd}_\mathcal{B}({\rm Ker}(\varepsilon_Y))=0$.

(2) Similarly as above, we have to study when $\eta_{X^\bullet}$ and $\varepsilon_{Y^\bullet}$ becomes isomorphisms in the singularity categories for any bounded complex $X^\bullet$ in $\mathcal{A}$ and any bounded complex $Y^\bullet$ in $\mathcal{B}$.

The following  short exact sequence of bounded complexes
$$0\longrightarrow X^\bullet \stackrel{\eta_{X^\bullet}} \longrightarrow GF(X^\bullet)\longrightarrow {\rm Cok}(\eta_{X^\bullet})\longrightarrow 0$$
yields an exact triangle in $\mathbf{D}_{\rm sg}(\mathcal{A})$
$$X^\bullet \stackrel{\eta_{X^\bullet}} \longrightarrow GF(X^\bullet)\longrightarrow {\rm Cok}(\eta_{X^\bullet})\longrightarrow \Sigma(X^\bullet).$$
Then $\eta_{X^\bullet}$ is an isomorphism if and only if ${\rm Cok}(\eta_{X^\bullet})$ is zero in $\mathbf{D}_{\rm sg}(\mathcal{A})$. The latter condition means exactly that ${\rm Cok}(\eta_{X^\bullet})$, as an object in $\mathbf{D}^b(\mathcal{A})$, is isomorphic to a bounded complex of projective objects for any bounded complex $X^\bullet$. As each component of ${\rm Cok}(\eta_{X^\bullet})$ is given by ${\rm Cok}(\eta_{X^i})$, we infer that the previous condition is equivalent to the condition that for each object $X\in \mathcal{A}$,  ${\rm Cok}(\eta_{X})$ has finite projective dimension. In summary, we have proved that each $\eta_{X^\bullet}$ is an isomorphism in $\mathbf{D}_{\rm sg}(\mathcal{A})$ if and only if each ${\rm Cok}(\eta_{X})$ has finite projective dimension.

Similarly, we can prove that each $\varepsilon_{Y^\bullet}$ is an isomorphism in $\mathbf{D}_{\rm sg}(\mathcal{B})$ if and only if ${\rm Ker}(\eta_{Y})$ has finite projective dimension for each $Y\in \mathcal{B}$. This completes the proof of (2).

We omit the proof of (3), as it is very similar to the one of (2).
\end{proof}

Let us mention a particular case.

\begin{cor}\label{cor:triequi}
 Let $F\colon \mathcal{A}\rightleftarrows \mathcal{B}\colon G$  be a Frobenius pair with both $F$ and $G$ faithful. Assume that ${\rm Cok}(\eta_X)$ and ${\rm Ker}(\varepsilon_Y)$ are projective for any $X\in \mathcal{A}$ and $Y\in \mathcal{B}$. Then $F$ and $G$ induce the following triangle equivalences
 $$\underline{\mathcal{GP}}(\mathcal{A})\simeq \underline{\mathcal{GP}}(\mathcal{B}), \; \mathbf{D}_{\rm sg}(\mathcal{A})\simeq \mathbf{D}_{\rm sg}(\mathcal{B}), \mbox{and } \mathbf{D}_{\rm def}(\mathcal{A})\simeq \mathbf{D}_{\rm def}(\mathcal{B}). $$
\end{cor}

For a finite dimension algebra $A$ over a field $k$, we denote by $A\mbox{-mod}$ the abelian category of finite dimensional left $A$-modules. The corresponding categories $\underline{\mathcal{GP}}(A\mbox{-mod})$, $\mathbf{D}_{\rm sg}(A\mbox{-mod})$ and $\mathbf{D}_{\rm def}(A\mbox{-mod})$ are usually denoted by $A\mbox{-}\underline{\rm Gproj}$, $\mathbf{D}_{\rm sg}(A)$ and $\mathbf{D}_{\rm def}(A)$, respectively.

\begin{exm}
{\rm Let $A$ and $B$ be two finite dimensional algebras over a field $k$. By \cite[Proposition~2.2 and Corollary~3.1]{DM07}, under very mild conditions, a stable equivalence of Morita type is  a stable equivalence of adjoint type \cite{Xi08}. It  yields a classical Frobenius pair between $A\mbox{-mod}$ and $B\mbox{-mod}$, which satisfies the conditions in Corollary~\ref{cor:triequi}. So, we have three induced triangle equivalences
 $$A\mbox{-}\underline{\rm Gproj}\simeq B\mbox{-}\underline{\rm Gproj}, \; \mathbf{D}_{\rm sg}(A)\simeq \mathbf{D}_{\rm sg}(B), \mbox{and } \mathbf{D}_{\rm def}(A)\simeq \mathbf{D}_{\rm def}(B). $$
For the leftmost equivalence, one might compare \cite[Section~4]{LX}; for the middle singular equivalence, one might compare \cite[Corollary~3.6]{ZZ}.}
\end{exm}

\appendix

\section{The global  Gorenstein dimension of an abelian category}

We fix an arbitrary abelian category $\mathcal{A}$ which has enough projective objects and enough injective objects. We will give a new and direct proof to the following known result \cite{BR}: the supermum of the Gorenstein projective dimension of all objects in $\mathcal{A}$ coincides with the supermum of the Gorenstein injective dimension of all objects in $\mathcal{A}$. The common value is called the global Gorenstein dimension of $\mathcal{A}$.

The above mentioned result appears first in \cite[Proposition~VII.1.3]{BR}, while its proof is indirect and seems nontrivial to follow.  The result for the module category is rediscovered in \cite{BM10} via a completely different method; compare \cite[Theorem~4.1]{Emm12}.

We will abbreviate ${\rm Gpd}_\mathcal{A}$ and ${\rm Gid}_\mathcal{A}$ as ${\rm Gpd}$ and ${\rm Gid}$, respectively.  Let us start with a well-known fact.

\begin{lem}\label{lemA:fin-inj}
Assume that $M\in \mathcal{A}$ satisfies ${\rm id}(M)<\infty$. Then the following statements hold.
 \begin{enumerate}
 \item ${\rm Ext}^i_\mathcal{A}(G, M)=0$ for any Gorenstein projective object $G$ and  $i>0$;
 \item ${\rm Gpd}(M)={\rm pd}(M)$.
 \end{enumerate}
\end{lem}

\begin{proof}
The Ext-vanishing in (1) is well known, which follows by a dimension-shift argument. For (2), it suffices to prove ${\rm pd}(M)\leq {\rm Gpd}(M)$. We may assume that ${\rm Gpd}(M)=n<\infty$. By \cite[Theorem~1.1]{AB}, there is a short exact sequence
$$0\longrightarrow M \longrightarrow Q \longrightarrow G\longrightarrow 0$$
with ${\rm pd}(Q)\leq n$ and $G$ Gorenstein projective. This sequence splits due to the Ext-vanishing condition in (1). Then the required inequality follows.
\end{proof}

Denote by ${\rm spdi}(\mathcal{A})$ the supremum of projective dimension of all injective objects in $\mathcal{A}$, and by ${\rm sidp}(\mathcal{A})$ the supremum of injective dimension of all projective objects in $\mathcal{A}$.

The following fact is also well known.

\begin{lem}\label{lemA:finite}
Assume that both $\mathrm{spdi}(\mathcal{A})$ and $\mathrm{sidp}(\mathcal{A})$ are finite. Then we have $\mathrm{spdi}(\mathcal{A})= \mathrm{sidp}(\mathcal{A})$.
\end{lem}

\begin{proof}
  By duality, we may assume on the contrary that $\mathrm{spdi}(\mathcal{A})<\mathrm{sidp}(\mathcal{A})=m$. Take a projective object $P$ satisfying ${\rm id}(P)=m$. Hence, we have ${\rm Ext}_\mathcal{A}^m(X, P)\neq 0$ for some object $X$. Take a short exact sequence
 $$0\longrightarrow X\longrightarrow I\longrightarrow X'\longrightarrow 0$$
 with $I$ injective. By ${\rm Ext}^{m+1}_\mathcal{A}(X', P)=0$,  the induced map
 $${\rm Ext}_\mathcal{A}^m(I, P)\longrightarrow {\rm Ext}^m_\mathcal{A}(X, P)$$
 is surjective. Consequently, we have ${\rm Ext}_\mathcal{A}^m(I, P)\neq 0$. This implies that ${\rm pd}(I)\geq m$, contradicting to the inequality.
\end{proof}

\begin{lem}\label{lemA:Goren}
Assume that ${\rm sidp}(\mathcal{A})<\infty$. Suppose that there is an exact sequence $0\rightarrow M\rightarrow P^0\rightarrow P^1\rightarrow \cdots$ with each $P^i$ projective. Then $M$ is Gorenstein projective.
\end{lem}

\begin{proof}
By assumption, each projective $P$ has finite injective dimension. Then by a dimension-shift argument, we have ${\rm Ext}_\mathcal{A}^i(M, P)=0$ for any $i\geq 1$. Due to the same reason, each kernel $K^j$ of $P^j\rightarrow P^{j+1}$ satisfies ${\rm Ext}_\mathcal{A}^i(K^j, P)=0$. By Lemma~\ref{lem:GP}(2), this implies that $M$ is Gorenstein projective.
\end{proof}

\begin{lem}\label{lemA:finite-G1}
 Assume that $\mathrm{sup} \{\mathrm{Gpd}(M) \; |\; M\in \mathcal{A} \} = n<\infty$. Then we have $\mathrm{spdi}(\mathcal{A})= \mathrm{sidp}(\mathcal{A}) \leq n$.
\end{lem}

\begin{proof}
For any injective object $I$, we have  $\mathrm{pd}(I) = \mathrm{Gpd}(I)$ by Lemma~\ref{lemA:fin-inj}. Then $\mathrm{spdi}(\mathcal{A}) \leq n$ follows from the assumption.

For any projective object $P$,  take an exact sequence
$$0\longrightarrow P\longrightarrow I^0 \longrightarrow \cdots \longrightarrow I^{n-1}\longrightarrow I^{n} \longrightarrow L \longrightarrow 0$$
with each $I^i$ injective. By the assumption, we have ${\rm Gpd}(L)\leq n$. It follows that ${\rm Ext}_\mathcal{A}^{n+1}(L, P)=0$. By a dimension-shift, we have
$${\rm Ext}_\mathcal{A}^{n+1}(L, P)\simeq {\rm Ext}^1_\mathcal{A}(L, L')$$
where $L'$ is the image of $I^{n-1}\rightarrow I^n$. Therefore, the short exact sequence
$$0\longrightarrow L'\longrightarrow I^n\longrightarrow L\longrightarrow 0$$
splits, and thus $L'$ is injective. This proves that ${\rm id}(P)\leq n$ and then ${\rm sidp}(\mathcal{A})\leq n$.  Then we are done by Lemma~\ref{lemA:finite}.\end{proof}

The following proof seems to be  new, where the argument is similar to the one in \cite[Section~3]{Chen10}.

\begin{lem}\label{lemA:finite-G2}
 Assume that $\mathrm{spdi}(\mathcal{A})\leq m $ and $\mathrm{sidp}(\mathcal{A})\leq  m$ for some integer $m$. Then we have  $\mathrm{sup}\{\mathrm{Gpd}(M)\; |\; M\in \mathcal{A} \} \leq m$.
\end{lem}

\begin{proof}
Fix an arbitrary object $M$. Consider an injective resolution
$$0\longrightarrow M\longrightarrow I^0 \stackrel{\partial^0} \longrightarrow I^1\stackrel{\partial^1}\longrightarrow I^1\longrightarrow \cdots $$
of $M$. For each $i\geq 0$, we choose a projective resolution
$$\pi^i\colon P^{i, \bullet}\longrightarrow I^i$$
such that $P^{i, \bullet}=(P^{i, j}, d_0^{i, j})_{j\leq 0}$ satisfies $P^{i, j}=0$ for $j\leq -(m+1)$. Here, we use the assumption ${\rm spdi}(\mathcal{A})\leq m$. By Comparison Theorem, there exists a cochain map
$$d_h^{i, \bullet}\colon P^{i,\bullet}\longrightarrow P^{i+1, \bullet}$$
extending the differential $\partial^i$. Set $d_{1}^{i,j}=(-1)^{j}d_h^{i,j}$.

The bigraded objects $P^{\bullet, \bullet}$ are endowed with two endomorphisms $d_0$ of degree (0,1) and $d_1$ of degree (1,0), which satisfy $d_0\circ d_0 = 0$ and $d_1\circ d_0 + d_0\circ d_1 = 0$. However, $d_1\circ d_1$ is not necessarily zero. Therefore,  $P^{\bullet, \bullet}$ is not a bicomplex in general.

The cochain map
$$d_h^{i+1,\bullet}\circ d_h^{i,\bullet}\colon  P^{i,\bullet}\longrightarrow P^{i+2,\bullet}$$
extends the zero map $0=\partial^{i+1}\circ \partial^i\colon I^i\rightarrow I^{i+2}$. Therefore, it is   is homotopic to zero. Hence the homotopy maps give rise an endomorphism $d_2$ of $P^{\bullet, \bullet}$ with degree $(2, -1)$ which satisfying
 $$d_0\circ  d_2 + d_1\circ d_1 + d_2\circ d_0 = 0.$$
 It is routine to check that
$d_1\circ d_2 + d_2\circ d_1$ commutes with $d_0$. In other words,
$$d_1^{i+2, \bullet-1}\circ  d_2^{i, \bullet} + d_2^{i+1, \bullet}\circ d_1^{i, \bullet}\colon P^{i, \bullet}\longrightarrow P^{i+3, \bullet}(-1)$$
 is a cochain map, where $(-1)$ denote the degree-shift functor of complexes (that is, $(-1)$ does not change the sign of the differentials). However, any cochain map $ P^{i, \bullet}\rightarrow P^{i+3, \bullet}(-1)$ is necessarily homotopic to zero. The homotopy for the above cochain map yields an endomorphism $d_3$ of $P^{\bullet, \bullet}$ of degree $(3, -2)$ satisfying
 $$d_0\circ  d_3 + d_1\circ d_2 +d_2\circ d_1 + d_3\circ d_0 = 0.$$

We repeat the above process to construct for each $l\geq 0$ an endomorphism $d_l$ on $P^{\bullet, \bullet}$ with degree $(l, -l+1)$ such that $\sum_{i=0}^{l} d_i\circ  d_{l-i} = 0$. The bigraded objects $P^{\bullet, \bullet}$ together with such endomorphisms $d_l$ becomes a  \emph{quasi-bicomplex} in the sense of \cite[p.2725]{Chen10}.

We form the \emph{total complex} $Q^\bullet$ of the quasi-bicomplex $P^{\bullet, \bullet}$ as follows: each component $Q^s$ is given by a finite direct sum $\oplus_{i+j=s}P^{i, j}$; the differential $Q^s\rightarrow Q^{s+1}$ is given by $\sum_{l\geq 0} d_l$, more precisely, its restriction to $P^{i, j}$ is given by $\sum_{l\geq 0} d_l^{i, j}$, where we observe that $d_l^{i, j}=0$ whenever $l\geq j+m+2$. By \cite[Proposition~3.4]{Chen10}, the morphisms $\pi^i\colon P^{i, 0}\rightarrow I^i$ induce a quasi-isomorphism
$$\pi^\bullet \colon Q^\bullet\longrightarrow I^\bullet.$$
In other words, the total complex $Q^\bullet$ is quasi-isomorphic to $M$, viewed as a stalk complex concentrated in degree zero.

By the very definition, the total complex $Q^\bullet$ is of the form
$$0\longrightarrow Q^{-m} \longrightarrow Q^{-m-1} \longrightarrow \cdots \longrightarrow Q^0\longrightarrow Q^1\longrightarrow \cdots$$
where each $Q^s$ is projective. As it is quasi-isomorphic to $M$, we obtain a short exact sequence
$$0\longrightarrow B^0\longrightarrow Z^0\longrightarrow M\longrightarrow 0,$$
where $B^0$ is the image of $Q^{-1}\rightarrow Q^0$ and $Z^0$ is the kernel of $Q^{0}\rightarrow Q^1$. Moreover, we observe that ${\rm pd}(B^0)\leq m-1$, and that $Z^0$ is Gorenstein projective by the assumption ${\rm sidp}(\mathcal{A})\leq m$ and  Lemma~\ref{lemA:Goren}. Then the above short exact sequence implies that ${\rm Gpd}(M)\leq m$, as required.
\end{proof}

The main result is as follows, which slightly reformulates the one in  \cite[Proposition~VII.1.3]{BR}.

\begin{thm}\label{thmA:balance}
Let $\mathcal{A}$ be an abelian category with enough projective objects and enough injective objects. Then we have
$$\mathrm{sup}\{\mathrm{Gpd}(M)\;  | \; M\in \mathcal{A} \} ={\rm max}\{{\rm spdi}(\mathcal{A}), {\rm sidp}(\mathcal{A}) \} = \mathrm{sup}\{\mathrm{Gid}(M)\;  |\;  M\in \mathcal{A} \}.$$
\end{thm}

This common value is call the \emph{global Gorenstein  dimension} of $\mathcal{A}$, denoted by ${\rm gl.Gdim}(\mathcal{A})$.

\begin{proof}
The inequality $\mathrm{sup}\{\mathrm{Gpd}(M)\;  | \; M\in \mathcal{A} \} \leq {\rm max}\{{\rm spdi}(\mathcal{A}), {\rm sidp}(\mathcal{A}) \}$ follows from Lemma~\ref{lemA:finite-G2}, while $ {\rm max}\{{\rm spdi}(\mathcal{A}), {\rm sidp}(\mathcal{A}) \}\leq \mathrm{sup}\{\mathrm{Gpd}(M)\;  | \; M\in \mathcal{A} \}$ follows from  Lemma~\ref{lemA:finite-G1}. This establishes the left equality. By duality, we have the right equality.
\end{proof}

Let $d\geq 0$. The abelian category $\mathcal{A}$ is \emph{$d$-Gorenstein} if ${\rm gl.Gdim}(\mathcal{A})\leq d$. In view of Theorem~\ref{thmA:balance} and Lemma~\ref{lemA:finite}, this is equivalent to $\mathrm{spdi}(\mathcal{A})= \mathrm{sidp}(\mathcal{A})\leq d$. The terminology is justified by the following fact: a ring $R$ is \emph{left $d$-Gorenstein} in the sense of \cite{Bel00} if and only if the abelian  category $R\mbox{-Mod}$ of left $R$-modules is $d$-Gorenstein; compare \cite[Theorem~6.9]{Bel00} and \cite[Definition~VII.2.5]{BR}.

In view of Lemma~\ref{lemA:finite}, we remind the following open question; see \cite[p.123]{BR}.
\vskip 5pt

\noindent {\bf Question.} \emph{Let $\mathcal{A}$ be an  abelian category with enough projective objects and enough injective objects. Does the equality $\mathrm{spdi}(\mathcal{A})= \mathrm{sidp}(\mathcal{A})$ hold always?}

\vskip 5pt

Let $A$ be an artin algebra over a commutative artinian ring. The famous \emph{Gorenstein symmetry conjecture} for $A$ states that the injective dimension of $A$ as the  left regular $A$-module  coincides with the injective dimension of $A$ as the right regular $A$-module.  We observe that the above question for the abelian category $A\mbox{-mod}$ of finitely generated left $A$-modules is equivalent to the Gorenstein symmetry conjecture for $A$.

\vskip 10pt

\noindent {\bf Acknowledgements.}\quad The authors thank Zhibin Zhao for helpful discussion; the second author thanks Professor Changchang Xi for informing him about the stable equivalence of adjoint type.  This work is supported by the National Natural Science Foundation of China (No.s 11671245, 11871125 and 11971449), Anhui Initiative in Quantum Information Technologies (AHY150200), and Natural Science Foundation of Chongqing (No. cstc2018jcyjAX0541).

After the authors put the first version of the paper on arXiv.org, Jiangsheng Hu kindly points out that Theorem~\ref{thm:GP} is partly obtained in a recent preprint \cite{HLLZ} via a slightly different proof; moreover, Theorem~\ref{thmA:balance} is generalized in \cite[Theorem~1.2]{LXG} to an abelian category with an admissible balanced pair in the sense of \cite{Chen13}. We are very grateful to him for the references and  helpful comments.

\bibliography{}

\vskip 10pt

 {\footnotesize \noindent Xiao-Wu Chen,\\
 Key Laboratory of Wu Wen-Tsun Mathematics, Chinese Academy of Sciences,\\
 School of Mathematical Sciences, University of Science and Technology of China, Hefei 230026, Anhui, PR China\\
 }

\vskip 5pt

{\footnotesize \noindent Wei Ren,\\
 School of Mathematical Sciences, Chongqing Normal University, Chongqing 401331, PR China\\
 }

\end{document}